\documentclass[11pt]{amsart}
\usepackage[dvipsnames,usenames]{color}
\usepackage{hyperref}
\usepackage{graphicx}
\usepackage{epsfig}
\usepackage[latin1]{inputenc}
\usepackage{amsmath}
\usepackage{amsfonts}
\usepackage{amssymb}
\usepackage{amsthm}
\usepackage{amscd}
\usepackage{verbatim}
\usepackage{caption}
\usepackage{subcaption}
\usepackage[percent]{overpic}
\usepackage{pinlabel}
\usepackage{stmaryrd}
\usepackage{enumerate, enumitem}
\usepackage{todonotes}
\usepackage{bm}
\usepackage{thmtools}
\usepackage{thm-restate}
\usepackage{tikz}
\usetikzlibrary{arrows, patterns}
\usetikzlibrary{decorations.pathreplacing}
\usepackage{verbatim}
\usetikzlibrary{cd}

\newtheorem{theorem}{Theorem}[section]

\newtheorem{lemma}[theorem]{Lemma}
\newtheorem{proposition}[theorem]{Proposition}

\newtheorem{corollary}[theorem]{Corollary}

\theoremstyle{definition}
\newtheorem{definition}[theorem]{Definition}

\theoremstyle{remark}
\newtheorem{remark}[theorem]{Remark}
\newtheorem{example}[theorem]{Example}

\def\F{\mathbb{F}}
\def\N{\mathbb{N}}

\def\Z{\mathbb{Z}}

\def\bbT{\mathbb{T}}

\def\s{\mathfrak s}

\DeclareMathOperator{\hfk}{HFK}
\DeclareMathOperator{\cfk}{CFK}
\DeclareMathOperator{\cf}{CF}

\DeclareMathOperator{\hfhat}{\widehat{HF}}
\DeclareMathOperator{\cfhat}{\widehat{CF}}
\DeclareMathOperator{\cfkhat}{\widehat{CFK}}

\DeclareMathOperator{\sgn}{sgn}

\newenvironment{lrcases}
{\begin{aligned}}
{\end{aligned}}

\author[L. Truong]{Linh Truong}
\thanks{The author was partially supported by NSF grant DMS-1606451.}
\address {Department of Mathematics, Columbia University, New York, NY 10027}
\email{ltruong@math.columbia.edu}

\numberwithin{equation}{section}

\title[A refinement of the surgery formula and concordance]{A refinement of the Ozsv\'ath-Szab\'o large integer surgery formula and knot concordance}

\begin{document}

\begin{abstract} 
We compute the knot Floer filtration induced by the $(n,1)$--cable of the meridian of a knot in the manifold obtained by large integer surgery along the knot. We give a formula in terms of the original knot Floer complex of the knot in $S^3$. As an application, we show that the concordance invariant $a_1(K)$ of Hom can equivalently be defined in terms of filtered maps on the Heegaard Floer homology groups induced by the two-handle attachment cobordism of surgery along a knot in $S^3$. 
\end{abstract}

\maketitle

\section{Introduction}


Let $S^3_t(K)$ denote the manifold constructed as Dehn surgery along $K \subset S^3$ with surgery coefficient $t$. In \cite{HFK} Ozsv\'ath and Szab\'o construct a chain homotopy equivalence between certain subquotient complexes of the full knot Floer chain complex $\cfk^\infty(S^3, K)$ and Heegaard Floer chain complexes $\cfhat(S^3_t(K), \s_m)$ for sufficiently large integers $t$ for each spin$^c$ structure $\s_m$. This equivalence is known as the \emph{large integer surgery formula}.

The meridian $\mu$ of $K$ naturally lies inside of the knot complement $S^3 \setminus K$ and the surgered manifold $S^3_t(K)$. The meridian $\mu$ induces a filtration on $\cfhat(S^3_t(K), \s_m)$ for each spin$^c$ structure $\s_m$. In \cite{hedden} Hedden gives a formula for the filtered complex $\cfkhat(S^3_t(K), \mu, \s_m)$ in terms of $\cfk^\infty(S^3, K)$ for sufficiently large $t$. As an application of this formula, Hedden computes the knot Floer homology of Whitehead doubles and the Ozsv\'ath-Szab\'o concordance invariant $\tau$ of Whitehead doubles. In \cite{hedden-kim-livingston} Hedden, Kim, and Livingston generalize Hedden's formula by computing the full knot Floer complex $\cfk^\infty(S^3_t(K), \mu, \s_m)$ in terms of $\cfk^\infty(S^3, K)$ for sufficiently large $t$. As an application to knot concordance, they show that the subgroup of topologically slice knots of the concordance group contains a $\mathbb{Z}_2^\infty$ subgroup. 

\begin{figure}[tbh!]
\begin{center}
\begin{overpic}[scale=.3]{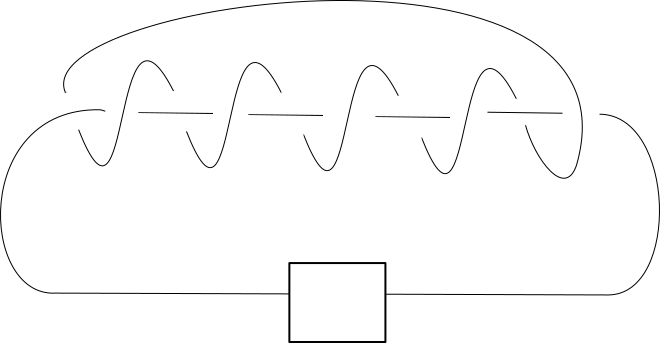}
\put (48,5) {$K$}
\put (8, 46) {$\mu_n$}
\end{overpic}
\end{center}
\caption[]{The two-component link $\mu_n$ and $K$ for $n = 5$}
\label{fig:link}
\end{figure}

We refine the theorems of Ozsv\'ath-Szab\'o, Hedden and Hedden-Kim-Livingston to determine the filtered chain homotopy type of $\cfk^\infty(S^3_t(K), \mu_n)$, where $\mu_n$ denotes the $(n,1)$--cable of the meridian of $K$, viewed as a knot in $S^3_t(K)$. See Figure~\ref{fig:link}. 
For each spin$^c$ structure $\s_m$, we show that the complex $\cfk^\infty(S^3_{t}(K), \mu_n, \s_m)$ is isomorphic to $\cfk^\infty(S^3,K)$, but endowed with a different $\Z \oplus \Z$ filtration and an overall shift in the homological grading.

\begin{theorem}
\label{theorem:infinityfiltration}
Let  $K$ be a knot in $S^3$ and fix $m, n \in \Z$. Then there exists $T = T(m, n) >0$ such that for all $t >T$,
the complex $\cfk^\infty(S^3_{t}(K), \mu_n, \mathfrak{s}_m)$ is isomorphic to $\cfk^\infty(S^3, K)[\epsilon]$ as an unfiltered complex, where $[\epsilon]$ denotes a grading shift that depends only on $m$ and $t$. 
Given a generator $[x,i,j]$ for $\cfk^\infty(S^3, K)$, the $\Z \oplus \Z$ filtration level of the same generator, viewed as a chain in $\cfk^\infty(S^3_{t}(K), \mu_n, \mathfrak{s}_m)$, is given by:

\begin{eqnarray*}
\mathcal{F}([x,i,j]) = \begin{cases}
[i,i] &\text{if } j \leq m+i
\\
[j-m, j-m-k] &\text{if } j = m+i +k, \ \text{where }1 \leq k < n
\\
[j-m, j-m-n] &\text{if } j \geq m+i +n
\end{cases}
\end{eqnarray*}
\end{theorem}


As a corollary, the $\Z$--filtered complex $\cfkhat(S^3_{t}(K), \mu_n, \s_m)$ is isomorphic to a subquotient complex of $\cfk^\infty(S^3,K)$, endowed with an $(n+1)$ step filtration $\mathcal{F}$:
\begin{eqnarray*}
0 \subseteq C_{\{ i < -n+1, j = m \}} 
\subseteq
\cdots 
\subseteq  C_{\{ i < 0, j = m \} } 
\subseteq C_{\{\max(i,j-m)=0\}} 
\end{eqnarray*}
This filtration is illustrated in Figure \ref{fig:filtration} in the case $n = 3$.
\begin{corollary}
\label{cor:hatfiltration}
Let $K \subset S^3$ be a knot, and fix $m, n \in \mathbb{Z}$. Then there exists $T = T(m, n) > 0$ such that for all $t > T$, the $\Z$--filtration on $\widehat{CF}(S^3_{t}(K), \mathfrak{s}_m)$ induced by $\mu_n \subset  S^3_t(K) $ is isomorphic to the filtered chain homotopy type of the $(n+1)$ step filtration on $C\{\max(i,j-m)=0\}$ described above. 
\end{corollary}

\begin{figure}[tbh!]
\begin{tikzpicture}[scale=.5]
	\fill[gray!12]  (5.4, 4.4) rectangle (-1, 3.6);
	\fill[gray!12]  (4.6, -.9) rectangle (5.4, 4.4);
	\draw[step=1, darkgray, very thin] (-.9, -.9) grid (5.5, 5.5);
	\begin{scope}[thin, black]
		\draw [->] (-1, 0) -- (5.6, 0);
		\draw [->] (5, -1) -- (5, 5.5);
	\end{scope}
	\begin{scope}[thick, densely dotted, blue]
		\draw [-] (4.6, 4.5) -- (-1, 4.5);
		\draw [-] (4.6, 3.5) -- (-1, 3.5);
		\draw [-] (4.6, 4.5) -- (4.6, 3.5);
	\end{scope}
	\begin{scope}[thick, densely dotted, red]
		\draw [-] (3.7, 4.6) -- (-1, 4.6);
		\draw [-] (3.7, 3.4) -- (-1, 3.4);
		\draw [-] (3.7, 4.6) -- (3.7, 3.4);
	\end{scope}
	\begin{scope}[thick, densely dotted, violet]
		\draw [-] (2.7, 4.7) -- (-1, 4.7);
		\draw [-] (2.7, 3.3) -- (-1, 3.3);
		\draw [-] (2.7, 4.7) -- (2.7, 3.3);
	\end{scope}
	
	\node [right] at (5,5) {$j$};
	\node [right] at (5.5,0) {$i$};
\end{tikzpicture}
\caption{$C\{\max(i,j-m)=0\}$ is the shaded region. The subregions bounded by the colored dots represent subcomplexes of the filtration $\mathcal{F}$ in the case $n = 3$. }
\label{fig:filtration}
\end{figure}
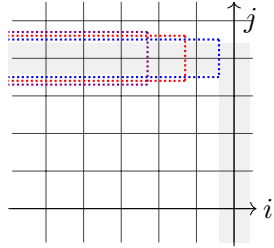

As an application, we show that the concordance invariant $a_1(K)$ of Hom \cite{Hom-concgroup} can equivalently be defined in terms of filtered maps on the Heegaard Floer homology groups induced by the two-handle attachment cobordism of surgery along a knot $K$ in $S^3$. The rationally null-homologous knot $\mu_n \subset S^3_t(K)$ induces a $\Z$-filtration of $\cfhat(S^3_t(K), \s_\tau)$ and $\cfhat(S^3_{-t}(K), \s_\tau)$, that is, a sequence of subcomplexes:
\[0 \subset \mathcal{F}_{bottom} \subset \mathcal{F}_{bottom+1} \subset \cdots \subset \mathcal{F}_{top-1} \subset \mathcal{F}_{top} = \cfhat(S^3_t(K), \s_\tau).\]
\[0 \subset \mathcal{F}'_{bottom} \subset \mathcal{F}'_{bottom+1} \subset \cdots \subset \mathcal{F}'_{top-1} \subset \mathcal{F}'_{top} = \cfhat(S^3_{-t}(K), \s_\tau).\]
Using the knot filtrations, an equivalent definition of $a_1(K)$ can be formulated in terms of the filtration $\mathcal{F}$ and $\mathcal{F}'$ induced by $\mu_n $ as a knot inside $S^3_t(K)$ and $S^3_{-t}(K)$.

\begin{theorem}
Let $n > 2g(K)$. 
For sufficiently large  surgery coefficient $t$, the concordance invariant $a_1(K)$ is equal to: \\
\label{a1}
\[ a_1(K) = \begin{cases}
\max \left\lbrace m \mid  \begin{lrcases}
 \cfhat(S^3_tK,\ s_\tau) / \mathcal{F}_{ \text{top}-1-m} \to \cfhat(S^3) 
\\
  \text{ induces a trivial map on homology }   
 \end{lrcases} \right \rbrace & \text{ if } \varepsilon(K) = -1, \\
\\
0 & \text{ if } \varepsilon(K) = 0, \\\\
\min \left\lbrace m \mid  \begin{lrcases}
 \cfhat(S^3)  \to  \mathcal{F}'_{ \text{bottom}+m} \subset \cfhat(S^3_{-t}K, \ \s_\tau) 
\\
  \text{ induces a trivial map on homology }   
 \end{lrcases} \right \rbrace & \text{ if } \varepsilon(K) = 1. 
\end{cases}
\]
\end{theorem}

This interpretation of the invariant $a_1(K)$ offers a  topological perspective that complements the original algebraic definition of $a_1(K)$. We will also include properties of the invariant $a_1(K)$ as well as computations of $a_1(K)$ for homologically thin knots and $L$--space knots.

\subsection*{Acknowledgements} 
The author thanks her advisors, Peter Ozsv\'ath and Zolt\'an Szab\'o, for their guidance. Adam Levine for reading the version of this work which appeared in the author's PhD thesis and for helpful comments. The author would also like to thank Matt Hedden, Jen Hom and Olga Plamenevskaya  for helpful conversations. 

\section{The knot Floer filtration of cables of the meridian in Dehn surgery along a knot}
\label{hfk_mu_n}
In this section we will refine the theorem of Ozsv\'ath-Szab\'o to determine the filtered chain homotopy type of the knot Floer complex of ($S^3_t(K)$, $\mu_n$). 

We begin by recalling the large integer surgery formula from Ozsv\'ath and Szab\'o \cite{HFK}. Let $(\Sigma_g$, $\alpha_1, \dots, \alpha_g$, $\gamma_1, \dots, \gamma_g$, $w$, $z)$ be a doubly-pointed Heegaard diagram for $\cfk^\infty(S^3,K)$, where 
\begin{itemize}
\item the curve $\gamma_g = \mu$ is a meridian of the knot $K$
\item the curve $\alpha_g$ is a longitude for $K$
\item there is a single intersection point in $\alpha_g \cap \gamma_g = x_0$
\item the basepoints $w$ and $z$ lie on either side of $\gamma_g$
\end{itemize}
Let $\beta = \{\gamma_1, \dots, \gamma_{g-1}, \lambda_t\}$ be the set of curves in $\gamma$, with $\gamma_g$ replaced by a longitude $\beta_g = \lambda_t$ winding $t$ times around $\mu$. Label the unique intersection point $\gamma_g \cap \beta_g=\theta$. The Heegaard triple diagram $(\Sigma, \alpha, \beta, \gamma, w, z)$ represents a cobordism between $S^3$ and $S^3_t(K)$. See Figure~\ref{fig:originalwindingregion}.

\begin{figure}[htb!]
\begin{tikzpicture}[scale=.75]
		\draw [-, thick] (-9, 0) -- (9, 0);
		\draw [-, thick] (-9,2 ) -- (9, 2);
		\draw [-, red, thick] (-9, 1) -- (9, 1);
		
		\draw [-, blue, thick] (-8,0 ) -- (-6, 2);
		\draw [-, blue, thick] (-6,0 ) -- (-4, 2);
		\draw [-, blue, thick] (-4,0 ) -- (-2, 2);
		\draw [-, blue, thick] (0, 1.5)  parabola (-2,0 );
		\draw [-, blue, thick] (0,1.5 )  parabola  (2, 2);
		\draw [-, blue, thick] (2,0 ) -- (4, 2);
		\draw [-, blue, thick] (4,0 ) -- (6, 2);
		\draw [-, blue, thick] (6,0 ) -- (8, 2);
		\draw [-, blue, thick] (9, .5)  parabola (8,0);
		\draw [-, blue, thick] (-9, 1.5)  parabola (-8,2);
		
		\draw [-, dotted, blue] (8,2) -- (8, 0);
		\draw [-, dotted, blue] (2,2) -- (2, 0);
		\draw [-, dotted, blue] (4,2) -- (4, 0);
		\draw [-, dotted, blue] (6,2) -- (6, 0);
		\draw [-, dotted, blue] (-2,2) -- (-2, 0);
		\draw [-, dotted, blue] (-4,2) -- (-4, 0);
		\draw [-, dotted, blue] (-6,2) -- (-6, 0);
		\draw [-, dotted, blue] (-8,2) -- (-8, 0);
		
		\draw [-, thick, orange] (-0.5, 0) -- (-0.5, 2);
		
		\filldraw (-1.15, 1) circle (1.5pt) node[below  ] {\tiny$x_{-1}$};
		 \filldraw (-3, 1) circle (1.5pt) node[below  ] {\tiny$x_{-2}$};
		 \filldraw (-5, 1) circle (1.5pt) node[below  ] {\tiny$x_{-3}$};
		 \filldraw (-7, 1) circle (1.5pt) node[below  ] {\tiny$x_{-4}$};
		 
		 \filldraw (3, 1) circle (1.5pt) node[below right] {\tiny$x_1$};
		 \filldraw (5, 1) circle (1.5pt) node[below right] {\tiny$x_2$};
		 \filldraw (7, 1) circle (1.5pt) node[below right] {\tiny$x_3$};
		 
		 \filldraw (-.5, 1) circle (1.5pt) node[below right] {\tiny$x_0$};
		 \filldraw (-0.5, 1.4) circle (1.5pt) node[ right] {\tiny$\theta$};
		 
		  \node [below left] at (-8.2, 1) {\tiny$\alpha_g$};
		  \node [below right] at (7.8, 2) {\tiny$\beta_g$};
		  \node [ right] at (-.6, .2) {\tiny$\gamma_g$};
		  
		  \filldraw (-.2, 1.8) circle (.5pt) node[ right] {\tiny$w$};
		  \filldraw (-1.2, 1.8) circle (.5pt) node[ right] {\tiny$z$};
\end{tikzpicture}
\caption[]{Local picture of the winding region of the Heegaard triple diagram $(\Sigma, \alpha, \beta, \gamma, w, z)$ for the cobordism between $S^3_tK$ and $S^3$}
\label{fig:originalwindingregion}
\end{figure}
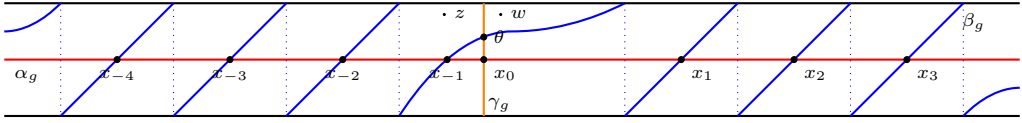

Let $C\{\text{max}(i, j-m)\} = 0$ denote the subquotient complex of $\text{CFK}^\infty(S^3, K)$ generated by triples $[x, i, j]$ with the $i$ and $j$ filtration levels satisfying the specified constraints. 
\begin{theorem}[\cite{HFK}] Let $K \subset S^3$ be a knot, and fix $m \in \mathbb{Z}$. Then there exists $T = T(m) > 0$ such that for all $t > T$, the chain map \[\Phi_m: \cfhat(S^3_t(K), \s_m) \to C\{\max(i, j-m) = 0\}\] defined by 
\[\Phi_m([x]) = \sum_{y \in T_\alpha \cap T_\gamma} \sum_{\{\psi\in \pi_2(x, \theta, y) \ | \ n_z(\psi) - n_w(\psi) =  m - \mathcal{F}(y), \ \mu(\psi) = 0 \} } [y, -n_w(\psi), m-n_z(\psi)]\]
induces an isomorphism of chain complexes. 
\end{theorem}
\begin{remark}
Here, as usual, the labeling of the spin$^c$ structures is determined by the condition that $\s_m$ can be extended over the cobordism $-W_t$ from $-S^3_t(K)$ to $-S^3$ associated to the two-handle addition along $K$ with framing $t$, yielding a spin$^c$ structure $\mathfrak{r}_m$ satisfying
\[ \langle c_1(\mathfrak{r}_m, [S]) \rangle + \  t = 2 m.\]
Above, $S$ denotes a surface in $W_t$ obtained from closing off a Seifert surface for $K$ in $S^3$ to produce a surface $S$ of square $t$.
\end{remark}

We refine the theorem of Ozsv\'ath-Szab\'o to determine the filtered chain homotopy type of the knot Floer complex of ($S^3_t(K)$, $\mu_n$). Consider the meridian $\mu = \mu_K$ of a knot $K$. The meridian $\mu$ naturally lies inside of the knot complement $S^3 \setminus K$ and the surgered manifold $S^3_t(K)$. For $n \in \N$, $\mu_n$ denotes the $(n,1)$--cable of $\mu_K$, and also lies inside $S^3 \setminus K$ and the surgered manifold $S^3_tK$. The knot $\mu_n$ is homologically equivalent to $n \cdot [\mu]$ in $H_1(S^3_t(K))$. When $n =1$, $\mu_1 = \mu$. See Figure~\ref{fig:link} for a picture of the two-component link $K \cup \mu_n$.

For all $n \geq 1$ there is a natural $(n+1)$-step algebraic filtration $\mathcal{F}$ on the subquotient complex $C_{\{\max(i,j-m)=0\}}$ of $\cfk^\infty(S^3,K)$:
\begin{eqnarray*}
0 \subseteq C_{\{ i < -n+1, j = m \}} 
\subseteq
\cdots 
\subseteq  C_{\{ i < 0, j = m \} } 
\subseteq C_{\{\max(i,j-m)=0\}}. 
\end{eqnarray*}
This filtration is illustrated in the case $n = 3$ in Figure~\ref{fig:filtration}.

Theorem \ref{hatfiltration} says that this algebraic filtration $\mathcal{F}$ corresponds to a relative $\mathbb{Z}$-filtration on $\cfhat(S^3_t(K), \s_m)$ induced by $\mu_n \in S^3_t(K)$. This generalizes work of Hedden \cite{hedden} who studied the $n=1$ case of the filtered complex $\cfkhat(S^3_t(K), \mu, \s_m)$.

\begin{theorem}
\label{hatfiltration}
Let $K \subset S^3$ be a knot, and fix $m, n \in \mathbb{Z}$. Then there exists $T = T(m, n) > 0$ such that for all $t > T$, the following holds: The filtered chain homotopy type of the $(n+1)$ step filtration $\mathcal{F}$ on $C\{\max(i,j-m)=0\}$ described above is  filtered chain homotopy equivalent to that of the filtration on $\widehat{CF}(S^3_{t}(K), \mathfrak{s}_m)$ induced by $\mu_n \subset  S^3_t(K) $. 

\end{theorem}

\begin{proof}
The key observation will be that the triple diagram $(\Sigma, \alpha, \beta, \gamma, w, z)$ used to define $\Phi_m$ not only specifies a Heegaard diagram for the knot $(S^3, K)$, but also a Heegaard diagram for the knot $(S^3_t (K), \mu_n)$ with the addition of a basepoint $z'$. 
Place an extra basepoint $z'=z_n$ so that it is $n$ regions away from the basepoint $w$ in the Heegaard triple diagram representing the cobordism between $S^3$ and $S^3_t(K)$ as in Figure~\ref{fig:windingregion}. (This can be accomplished if $t$ is sufficiently large, e.g. if $t > 2n$). The knot represented by the doubly-pointed Heegaard diagram $(\Sigma, \alpha, \beta, w, z_n)$ is $\mu_n $ in $S^3_t(K)$. 

An intersection point $x' \in \bbT_\alpha \cap \bbT_\beta$ is said to be supported in the winding region if the component of $x'$ in $\alpha_g$ lies in the local picture of Figure~\ref{fig:windingregion}. Intersection points in the winding region are in $t$ to 1 correspondence with intersection points $x$ in $\bbT_\alpha \cap \bbT_\gamma$. 

Fix a Spin$^c$ structure $\s_m$ where $m \in \mathbb{Z}$. For $t$ (the surgery coefficient) sufficiently large, any generator 
$x' \in \bbT_\alpha \cap \bbT_\beta$ representing $\text{Spin}^c$ structure $\s_m$ is supported in the winding region. In this case, there is a uniquely determined $x \in \bbT_\alpha \cap \bbT_\gamma$ and a canonical small triangle $\psi \in \pi_2(x, \theta, x')$.

Suppose $\psi \in \pi_2(x, \theta, x')$ is the canonical small triangle and 
$x'\in \bbT_\alpha \cap \bbT_\beta$ is a generator representing $\text{Spin}^c$ structure $\s_m$. If  $k=n_z(\psi) \geq 0$ (so $n_w(\psi) = 0$), then the $\alpha_g$ component of $x'$ is $x_{k}$ (and lies $k$ units to the left of $x_0$) in Figure \ref{fig:windingregion}. In this case, $\Phi_m$ maps $x'$ to $C\{i=0,j\leq m\}$. 
On the other hand, if $x'$ is a generator with $n_z(\psi) = 0$ and $l = n_w(\psi) > 0$, then the $\alpha_g$ component of $x'$ is $x_{-l}$ (and lies $l$ steps to the right of $x_0$) in Figure \ref{fig:windingregion}. In this case, $\Phi_m$ maps $x'$ to the subcomplex $C\{i\leq -l,j = m\} \subset C\{i<0,j = m\}$.

\begin{figure}[htb!]
\begin{tikzpicture}[scale=.65]
		\draw [-, thick] (-9, 0) -- (11, 0);
		\draw [-, thick] (-9,2 ) -- (11, 2);
		\draw [-, red, thick] (-9, 1) -- (11, 1);
		
		\draw [-, blue, thick] (-8,0 ) -- (-6, 2);
		\draw [-, blue, thick] (-6,0 ) -- (-4, 2);
		\draw [-, blue, thick] (-4,0 ) -- (-2, 2);
		\draw [-, blue, thick] (0, 1.5)  parabola (-2,0 );
		\draw [-, blue, thick] (0,1.5 )  parabola  (2, 2);
		\draw [-, blue, thick] (2,0 ) -- (4, 2);
		\draw [-, blue, thick] (4,0 ) -- (6, 2);
		\draw [-, blue, thick] (6,0 ) -- (8, 2);
		\draw [-, blue, thick] (8,0 ) -- (10, 2);
		\draw [-, blue, thick] (11, .5)  parabola (10,0);
		\draw [-, blue, thick] (-9, 1.5)  parabola (-8,2);
		
		\draw [-, dotted, blue] (10,2) -- (10, 0);
		\draw [-, dotted, blue] (8,2) -- (8, 0);
		\draw [-, dotted, blue] (2,2) -- (2, 0);
		\draw [-, dotted, blue] (4,2) -- (4, 0);
		\draw [-, dotted, blue] (6,2) -- (6, 0);
		\draw [-, dotted, blue] (-2,2) -- (-2, 0);
		\draw [-, dotted, blue] (-4,2) -- (-4, 0);
		\draw [-, dotted, blue] (-6,2) -- (-6, 0);
		\draw [-, dotted, blue] (-8,2) -- (-8, 0);
		
		\draw [-, thick, orange] (-0.5, 0) -- (-0.5, 2);
		
		\filldraw (-1.15, 1) circle (1.5pt) node[below  ] {\tiny$x_{1}$};
		 \filldraw (-3, 1) circle (1.5pt) node[below  ] {\tiny$x_{2}$};
		 \filldraw (-5, 1) circle (1.5pt) node[below  ] {\tiny$x_{3}$};
		 \filldraw (-7, 1) circle (1.5pt) node[below  ] {\tiny$x_{4}$};
		 
		 \filldraw (3, 1) circle (1.5pt) node[below right] {\tiny$x_{-1}$};
		 \filldraw (5, 1) circle (1.5pt) node[below right] {\tiny$x_{-2}$};
		 \filldraw (7, 1) circle (1.5pt) node[below right] {\tiny$x_{-3}$};
		 \filldraw (9, 1) circle (1.5pt) node[below right] {\tiny$x_{-4}$};
		 
		 \filldraw (-.5, 1) circle (1.5pt) node[below right] {\tiny$x_0$};
		 \filldraw (-0.5, 1.4) circle (1.5pt) node[ right] {\tiny$\theta$};
		 
		  \node [below left] at (-8.2, 1.1) {\tiny$\alpha_g$};
		  \node [below right] at (9.8, 2) {\tiny$\beta_g$};
		  \node [ right] at (-.6, .2) {\tiny$\gamma_g$};
		  
		  \filldraw (-.2, 1.8) circle (.5pt) node[ right] {\tiny$w$};
		  \filldraw (-1.2, 1.8) circle (.5pt) node[ right] {\tiny$z$};
		  \filldraw (6.5, 1.5) circle (.5pt) node[ right] {\tiny$z_3$};
\end{tikzpicture}

\caption[]{Local picture of the winding region of the Heegaard triple diagram $(\Sigma, \alpha, \beta, \gamma, w, z_n)$ for the cobordism between $S^3_tK$ and $S^3$. The basepoint $z_n$ is located $n$ regions away from the basepoint $w$ in the Heegaard  diagram $(\Sigma, \alpha, \beta, w, z_n)$. Here we depict the basepoint $z_n$ for $n=3$.
}
\label{fig:windingregion} 
\end{figure}
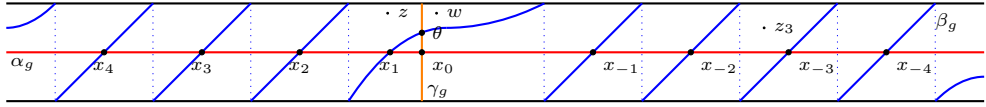
The following lemma (which generalizes Lemma 4.2 of \cite{hedden}) will be used to finish the proof.

\begin{lemma}
\label{lem:filtration}
Let $p \in \cfkhat(S^3_t(K), \mu_n, \s_m)$ be a generator supported in the winding region, and let $x_i$ denote the $\alpha_g$ component of the corresponding intersection point in $\bbT_\alpha \cap \bbT_\beta$, where the $x_i$ are labeled as in Figure~\ref{fig:windingregion}. Then
\[
\mathcal{F}(p) = 
\begin{cases}
\mathcal{F}_{top} & i > 0;
\\
\mathcal{F}_{top + i} & -n < i < 0 ;
\\
\mathcal{F}_{bottom} & i \leq -n.
\end{cases}
\]
Here, $\mathcal{F}_{top}$ (respectively, $\mathcal{F}_{bottom}$) denotes the top (respectively, bottom) filtration level of $\cfkhat(S^3_t(K), \mu_n, \s_m)$. $\mathcal{F}_{top - i}$ denotes the filtration level that is $i$ lower than $\mathcal{F}_{top}$. In addition $\mathcal{F}_{bottom} = \mathcal{F}_{top-n}$, so this is an $(n+1)$-step filtration.
\begin{proof}
The $\mathbb{Z}$-filtration $\mathcal{F}$ is defined by the relative \emph{Alexander} grading $A'_n$ induced by $\mu_n$ on $\cf^\infty(S^3_tK, \s_m)$. That is, 
\[\mathcal{F}(p) - \mathcal{F}(q) = n_{z_n}(\phi) - n_w(\phi)\] 
where $\phi \in \pi_2(p, q)$ is a Whitney disk connecting $p, q \in \bbT_\alpha \cap \bbT_\beta$.

Let $p, q \in \cfkhat(S^3_tK, \mu_n, \s_m)$ be generators supported in the
winding region, and let $x_i$, $x_j$ denote the $\alpha_g$ components of the corresponding intersection points $\bbT_\alpha \cap \bbT_\beta$. Assume without loss of generality that $i < j$ (so that $x_i$ lies to the right of $x_j$). 

We will define a set of $n$ arcs $\delta_1, \dots, \delta_n$ on $\beta$ as follows.
Let $\delta_1$ denote the arc on $\beta$ connecting $x_{1} $ to $x_{-1}$. 
Let $\delta_k$ denote on the arc on $\beta$ connecting $x_{-(k-1)}$ to $x_{-k}$, for $k \in \{2, \dots, n\}$. 

We will construct a Whitney disk $\phi_{p,q} \in \pi_2(p,q)$ with the following properties:
\begin{itemize}
\item If $i >0$ and $j > 0$, (that is, $x_i$, $x_j$ both lie on the left of $x_0$), then 
$\partial \phi_{p,q} $ doesn't contain any arc $\delta_k$. Therefore, 
\[\mathcal{F}(p) - \mathcal{F}(q) = 0. \]


\item If $i \leq -n$ and $j \leq -n$, (that is, $x_i$, $x_j$ both lie $\geq n$ steps to the right of $x_0$), then
$\partial \phi_{p,q} $ doesn't contain any arc $\delta_k$. Therefore, 
\[\mathcal{F}(p) - \mathcal{F}(q) = 0. \]
\item 
If $i < -n$ and $j > 0$, (that is, $x_j$ lies to the left of $x_0$ and $x_i$ lies $i$ steps to the right of $x_0$), then $\partial \phi_{p,q} $ contains the $n$ arcs $\delta_1, \dots, \delta_n$, each with multiplicity one. Therefore, 
\[\mathcal{F}(p) - \mathcal{F}(q) =  -n. \]

\item 
If $-n \leq i < 0$ and $j > 0$, (that is, $x_j$ lies to the left of $x_0$ and $x_i$ lies $i$ steps to the right of $x_0$), then $\partial \phi_{p,q} $ contains the $i$ arcs $\delta_1, \dots, \delta_i$, each with multiplicity one. Moreover, $\partial \phi_{p,q} $ doesn't contain the arcs $\delta_{k}$ for $k > i$. Therefore, 
\[\mathcal{F}(p) - \mathcal{F}(q) =  -i. \]


\item If $-n < j < 0$ and $i \leq -n$, (that is, $x_i$ lies $\geq n$ steps to the right of $x_0$ and $x_j$ lies $j$ steps to the right of $x_0$), then $\partial \phi_{p,q} $ contains the $n+j$ arcs $\delta_{|j|+1}, \dots, \delta_n$, each with multiplicity one. Moreover, $\partial \phi_{p,q} $ doesn't contain the arcs $\delta_{k}$ for $k \leq |j|$. Therefore, 
\[\mathcal{F}(p) - \mathcal{F}(q) = -n-j. \]

\item If $-n < i < 0$ and $-n < j < 0$, (that is, $x_j$ lies $ j$ steps to the right of $x_0$ and $x_i$ lies $i$ steps to the right of $x_0$), then $\partial \phi_{p,q} $ contains the $j-i$ arcs $\delta_{|j|+1}, \dots, \delta_{|i|}$, each with multiplicity one. Therefore,
\[\mathcal{F}(p) - \mathcal{F}(q) = i-j. \]

\end{itemize} 

Assuming the existence of such $\phi_{p,q}$, the lemma follows immediately. 

In \cite[Lemma 4.2]{hedden} Hedden constructs a Whitney disk $\phi_{p,q} \in \pi_2(p,q)$. The above enumerated properties of $\partial \phi_{p,q}$ will be immediate from the construction. We restate his construction here.
Note first since $p, q$ lie in the winding region, they correspond uniquely 
to intersection points $\tilde{p}$, $\tilde{q} \in \bbT_\alpha \cap \bbT_\gamma$. These intersection 
points $\tilde{p}$, $\tilde{q}$ can be connected by a Whitney disk
$\phi \in \pi_2(\tilde{p}, \tilde{q})$ with 
$n_w(\phi) = 0$ and $n_z(\phi) = k$ for some $k \in \Z_{\geq 0}$. 
This means that $\partial \phi$
contains $\gamma_g$ with multiplicity $k$, which further implies that the distance 
between $x_i$ and $x_j$ is $k$, that is, $i-j = k$. 
The domain of $\phi_{p,q}$ can then be obtained from the domain of $\phi$ by a simple modification in the winding region as described in \cite{hedden}. This modification is shown in Figure~\ref{fig:domainsinwindingregion}. It replaces the boundary component $k \cdot \gamma_g$ by a simple closed curve from an arc connecting $x_i$ and $x_j$ 
along $\alpha_g$ followed by an arc connecting
$x_j$ to $x_i$ along $\beta_g$, and which wraps $k$ times around the neck of the winding region.
\end{proof}
\end{lemma}

This completes the description of the knot Floer complex $\cfkhat(S^3_t(K), \mu_n)$ in terms of the complex $\cfk^\infty(S^3, K)$. 
\end{proof}

\begin{figure}[htb!]
\begin{subfigure}[]{\textwidth}
\caption{The domain of a disk $\phi \in \pi_2(\tilde{p}, \tilde{q})$.}
\centering
\begin{tikzpicture}[scale=.65]
		\fill[gray!15]    (-9,2) -- (-9,0) -- (-.5,0) -- (-.5,2);
		
		\draw [-, thick] (-9, 0) -- (11, 0);
		\draw [-, thick] (-9,2 ) -- (11, 2);
		\draw [-, red, thick] (-9, 1) -- (11, 1);
		
		\draw [-, blue, thick] (-8,0 ) -- (-6, 2);
		\draw [-, blue, thick] (-6,0 ) -- (-4, 2);
		\draw [-, blue, thick] (-4,0 ) -- (-2, 2);
		\draw [-, blue, thick] (0, 1.5)  parabola (-2,0 );
		\draw [-, blue, thick] (0,1.5 )  parabola  (2, 2);
		\draw [-, blue, thick] (2,0 ) -- (4, 2);
		\draw [-, blue, thick] (4,0 ) -- (6, 2);
		\draw [-, blue, thick] (6,0 ) -- (8, 2);
		\draw [-, blue, thick] (8,0 ) -- (10, 2);
		\draw [-, blue, thick] (11, .5)  parabola (10,0);
		\draw [-, blue, thick] (-9, 1.5)  parabola (-8,2);
		
		\draw [-, dotted, blue] (10,2) -- (10, 0);
		\draw [-, dotted, blue] (8,2) -- (8, 0);
		\draw [-, dotted, blue] (2,2) -- (2, 0);
		\draw [-, dotted, blue] (4,2) -- (4, 0);
		\draw [-, dotted, blue] (6,2) -- (6, 0);
		\draw [-, dotted, blue] (-2,2) -- (-2, 0);
		\draw [-, dotted, blue] (-4,2) -- (-4, 0);
		\draw [-, dotted, blue] (-6,2) -- (-6, 0);
		\draw [-, dotted, blue] (-8,2) -- (-8, 0);
		
		\draw [-, thick, orange] (-0.5, 0) -- (-0.5, 2);
		
		\filldraw (-1.15, 1) circle (1.5pt) node[below  ] {\tiny$x_{1}$};
		 \filldraw (-3, 1) circle (1.5pt) node[below  ] {\tiny$x_{2}$};
		 \filldraw (-5, 1) circle (1.5pt) node[below  ] {\tiny$x_{3}$};
		 \filldraw (-7, 1) circle (1.5pt) node[below  ] {\tiny$x_{4}$};
		 
		 \filldraw (3, 1) circle (1.5pt) node[below right] {\tiny$x_{-1}$};
		 \filldraw (5, 1) circle (1.5pt) node[below right] {\tiny$x_{-2}$};
		 \filldraw (7, 1) circle (1.5pt) node[below right] {\tiny$x_{-3}$};
		 \filldraw (9, 1) circle (1.5pt) node[below right] {\tiny$x_{-4}$};
		 
		 \filldraw (-.5, 1) circle (1.5pt) node[below right] {\tiny$x_0$};
		 \filldraw (-0.5, 1.4) circle (1.5pt) node[ right] {\tiny$\theta$};
		 
		  \node [below left] at (-8.2, 1.1) {\tiny$\alpha_g$};
		  \node [below right] at (9.8, 2) {\tiny$\beta_g$};
		  \node [ right] at (-.6, .2) {\tiny$\gamma_g$};
		  
		  \filldraw (-.2, 1.8) circle (.5pt) node[ right] {\tiny$w$};
		  \filldraw (-1.2, 1.8) circle (.5pt) node[ right] {\tiny$z$};
		  \filldraw (6.5, 1.5) circle (.5pt) node[ right] {\tiny$z_3$};

\end{tikzpicture}
\label{subfig:a}
\end{subfigure}
\\
\vspace{4mm}
\begin{subfigure}[]{\textwidth}
\caption{$\phi_{p,q} \in \pi_2(p,q)$ where $p$, $q$ have $\alpha_g$ components  $x_{-3}$, $x_{-1}$. $\partial \phi_{p,q}$ contains arcs $\delta_2$ and $\delta_3$ on $\beta$ drawn in violet.}
\centering
\begin{tikzpicture}[scale=.65]
		\fill[gray!30]    (-9,2) -- (-9,0) -- (2,0) -- (4,2);
		\fill[gray!15]    (4,2) -- (3,1) -- (5,1) -- (6,2);
		\fill[gray!30]    (3,1) -- (2,0) -- (4,0) -- (5,1);
		\fill[gray!15]    (5,1) -- (4,0) -- (6,0) -- (7,1);
		
		\draw [-, thick] (-9, 0) -- (11, 0);
		\draw [-, thick] (-9,2 ) -- (11, 2);
		\draw [-, red, thick] (-9, 1) -- (11, 1);
		
		\draw [-, blue, thick] (-8,0 ) -- (-6, 2);
		\draw [-, blue, thick] (-6,0 ) -- (-4, 2);
		\draw [-, blue, thick] (-4,0 ) -- (-2, 2);
		\draw [-, blue, thick] (0, 1.5)  parabola (-2,0 );
		\draw [-, blue, thick] (0,1.5 )  parabola  (2, 2);
		\draw [-, blue, thick] (2,0 ) -- (3, 1);
		\draw [-, violet, very thick] (3,1 ) -- (4, 2);
		\draw [-, violet, very thick] (4,0 ) -- (6, 2);
		\draw [-, violet, very thick] (6,0 ) -- (7, 1);
		\draw [-, blue, thick] (7,1 ) -- (8, 2);
		\draw [-, blue, thick] (8,0 ) -- (10, 2);
		\draw [-, blue, thick] (11, .5)  parabola (10,0);
		\draw [-, blue, thick] (-9, 1.5)  parabola (-8,2);
		
		\draw [-, dotted, blue] (10,2) -- (10, 0);
		\draw [-, dotted, blue] (8,2) -- (8, 0);
		\draw [-, dotted, blue] (2,2) -- (2, 0);
		\draw [-, dotted, blue] (4,2) -- (4, 0);
		\draw [-, dotted, blue] (6,2) -- (6, 0);
		\draw [-, dotted, blue] (-2,2) -- (-2, 0);
		\draw [-, dotted, blue] (-4,2) -- (-4, 0);
		\draw [-, dotted, blue] (-6,2) -- (-6, 0);
		\draw [-, dotted, blue] (-8,2) -- (-8, 0);
		
		\draw [-, thick, orange] (-0.5, 0) -- (-0.5, 2);
		
		\filldraw (-1.15, 1) circle (1.5pt) node[below  ] {\tiny$x_{1}$};
		 \filldraw (-3, 1) circle (1.5pt) node[below  ] {\tiny$x_{2}$};
		 \filldraw (-5, 1) circle (1.5pt) node[below  ] {\tiny$x_{3}$};
		 \filldraw (-7, 1) circle (1.5pt) node[below  ] {\tiny$x_{4}$};
		 
		 \filldraw (3, 1) circle (1.5pt) node[below right] {\tiny$x_{-1}$};
		 \filldraw (5, 1) circle (1.5pt) node[below right] {\tiny$x_{-2}$};
		 \filldraw (7, 1) circle (1.5pt) node[below right] {\tiny$x_{-3}$};
		 \filldraw (9, 1) circle (1.5pt) node[below right] {\tiny$x_{-4}$};
		 
		 \filldraw (-.5, 1) circle (1.5pt) node[below right] {\tiny$x_0$};
		 \filldraw (-0.5, 1.4) circle (1.5pt) node[ right] {\tiny$\theta$};
		 
		  \node [below left] at (-8.2, 1.1) {\tiny$\alpha_g$};
		  \node [below right] at (9.8, 2) {\tiny$\beta_g$};
		  \node [ right] at (-.6, .2) {\tiny$\gamma_g$};
		  
		  \filldraw (-.2, 1.8) circle (.5pt) node[ right] {\tiny$w$};
		  \filldraw (-1.2, 1.8) circle (.5pt) node[ right] {\tiny$z$};
		  \filldraw (6.5, 1.5) circle (.5pt) node[ right] {\tiny$z_3$};

\end{tikzpicture}
\label{subfig:b}
\end{subfigure}
\\
\vspace{4mm}
\begin{subfigure}[]{\textwidth}
\subcaption{$\phi_{p,q} \in \pi_2(p,q)$ where $p$, $q$ have $\alpha_g$ components  $x_{-2}$, $x_{-1}$.}
\centering
\begin{tikzpicture}[scale=.65]
		\fill[gray!15]    (-9,2) -- (-9,0) -- (2,0) -- (4,2);
		\fill[gray!15]    (3,1) -- (2,0) -- (4,0) -- (5,1);
		
		\draw [-, thick] (-9, 0) -- (11, 0);
		\draw [-, thick] (-9,2 ) -- (11, 2);
		\draw [-, red, thick] (-9, 1) -- (11, 1);
		
		\draw [-, blue, thick] (-8,0 ) -- (-6, 2);
		\draw [-, blue, thick] (-6,0 ) -- (-4, 2);
		\draw [-, blue, thick] (-4,0 ) -- (-2, 2);
		\draw [-, blue, thick] (0, 1.5)  parabola (-2,0 );
		\draw [-, blue, thick] (0,1.5 )  parabola  (2, 2);
		\draw [-, blue, thick] (2,0 ) -- (4, 2);
		\draw [-, blue, thick] (4,0 ) -- (6, 2);
		\draw [-, blue, thick] (6,0 ) -- (8, 2);
		\draw [-, blue, thick] (8,0 ) -- (10, 2);
		\draw [-, blue, thick] (11, .5)  parabola (10,0);
		\draw [-, blue, thick] (-9, 1.5)  parabola (-8,2);
		
		\draw [-, dotted, blue] (10,2) -- (10, 0);
		\draw [-, dotted, blue] (8,2) -- (8, 0);
		\draw [-, dotted, blue] (2,2) -- (2, 0);
		\draw [-, dotted, blue] (4,2) -- (4, 0);
		\draw [-, dotted, blue] (6,2) -- (6, 0);
		\draw [-, dotted, blue] (-2,2) -- (-2, 0);
		\draw [-, dotted, blue] (-4,2) -- (-4, 0);
		\draw [-, dotted, blue] (-6,2) -- (-6, 0);
		\draw [-, dotted, blue] (-8,2) -- (-8, 0);
		
		\draw [-, thick, orange] (-0.5, 0) -- (-0.5, 2);
		
		\filldraw (-1.15, 1) circle (1.5pt) node[below  ] {\tiny$x_{1}$};
		 \filldraw (-3, 1) circle (1.5pt) node[below  ] {\tiny$x_{2}$};
		 \filldraw (-5, 1) circle (1.5pt) node[below  ] {\tiny$x_{3}$};
		 \filldraw (-7, 1) circle (1.5pt) node[below  ] {\tiny$x_{4}$};
		 
		 \filldraw (3, 1) circle (1.5pt) node[below right] {\tiny$x_{-1}$};
		 \filldraw (5, 1) circle (1.5pt) node[below right] {\tiny$x_{-2}$};
		 \filldraw (7, 1) circle (1.5pt) node[below right] {\tiny$x_{-3}$};
		 \filldraw (9, 1) circle (1.5pt) node[below right] {\tiny$x_{-4}$};
		 
		 \filldraw (-.5, 1) circle (1.5pt) node[below right] {\tiny$x_0$};
		 \filldraw (-0.5, 1.4) circle (1.5pt) node[ right] {\tiny$\theta$};
		 
		  \node [below left] at (-8.2, 1.1) {\tiny$\alpha_g$};
		  \node [below right] at (9.8, 2) {\tiny$\beta_g$};
		  \node [ right] at (-.6, .2) {\tiny$\gamma_g$};
		  
		  \filldraw (-.2, 1.8) circle (.5pt) node[ right] {\tiny$w$};
		  \filldraw (-1.2, 1.8) circle (.5pt) node[ right] {\tiny$z$};
		  \filldraw (6.5, 1.5) circle (.5pt) node[ right] {\tiny$z_3$};
\end{tikzpicture}
\label{subfig:c}
\end{subfigure}
\\
\vspace{4mm}
\begin{subfigure}[]{\textwidth}
\caption{$\phi_{p,q} \in \pi_2(p,q)$ where $p$, $q$ have $\alpha_g$ components  $x_{1}$, $x_{2}$.}
\centering
\begin{tikzpicture}[scale=.65]
		\fill[gray!15]    (-9,2) -- (-9,0) -- (-4,0) -- (-2,2);
		\fill[gray!15]    (-3,1) -- (-4,0) -- (-2,0) -- (-1.15,1);

		\draw [-, thick] (-9, 0) -- (11, 0);
		\draw [-, thick] (-9,2 ) -- (11, 2);
		\draw [-, red, thick] (-9, 1) -- (11, 1);
		
		\draw [-, blue, thick] (-8,0 ) -- (-6, 2);
		\draw [-, blue, thick] (-6,0 ) -- (-4, 2);
		\draw [-, blue, thick] (-4,0 ) -- (-2, 2);
		\draw [-, blue, thick] (0, 1.5)  parabola (-2,0 );
		\draw [-, blue, thick] (0,1.5 )  parabola  (2, 2);
		\draw [-, blue, thick] (2,0 ) -- (4, 2);
		\draw [-, blue, thick] (4,0 ) -- (6, 2);
		\draw [-, blue, thick] (6,0 ) -- (8, 2);
		\draw [-, blue, thick] (8,0 ) -- (10, 2);
		\draw [-, blue, thick] (11, .5)  parabola (10,0);
		\draw [-, blue, thick] (-9, 1.5)  parabola (-8,2);
		
		\draw [-, dotted, blue] (10,2) -- (10, 0);
		\draw [-, dotted, blue] (8,2) -- (8, 0);
		\draw [-, dotted, blue] (2,2) -- (2, 0);
		\draw [-, dotted, blue] (4,2) -- (4, 0);
		\draw [-, dotted, blue] (6,2) -- (6, 0);
		\draw [-, dotted, blue] (-2,2) -- (-2, 0);
		\draw [-, dotted, blue] (-4,2) -- (-4, 0);
		\draw [-, dotted, blue] (-6,2) -- (-6, 0);
		\draw [-, dotted, blue] (-8,2) -- (-8, 0);
		
		\draw [-, thick, orange] (-0.5, 0) -- (-0.5, 2);
		
		\filldraw (-1.15, 1) circle (1.5pt) node[below  ] {\tiny$x_{1}$};
		 \filldraw (-3, 1) circle (1.5pt) node[below  ] {\tiny$x_{2}$};
		 \filldraw (-5, 1) circle (1.5pt) node[below  ] {\tiny$x_{3}$};
		 \filldraw (-7, 1) circle (1.5pt) node[below  ] {\tiny$x_{4}$};
		 
		 \filldraw (3, 1) circle (1.5pt) node[below right] {\tiny$x_{-1}$};
		 \filldraw (5, 1) circle (1.5pt) node[below right] {\tiny$x_{-2}$};
		 \filldraw (7, 1) circle (1.5pt) node[below right] {\tiny$x_{-3}$};
		 \filldraw (9, 1) circle (1.5pt) node[below right] {\tiny$x_{-4}$};
		 
		 \filldraw (-.5, 1) circle (1.5pt) node[below right] {\tiny$x_0$};
		 \filldraw (-0.5, 1.4) circle (1.5pt) node[ right] {\tiny$\theta$};
		 
		  \node [below left] at (-8.2, 1.1) {\tiny$\alpha_g$};
		  \node [below right] at (9.8, 2) {\tiny$\beta_g$};
		  \node [ right] at (-.6, .2) {\tiny$\gamma_g$};
		  
		  \filldraw (-.2, 1.8) circle (.5pt) node[ right] {\tiny$w$};
		  \filldraw (-1.2, 1.8) circle (.5pt) node[ right] {\tiny$z$};
		  \filldraw (6.5, 1.5) circle (.5pt) node[ right] {\tiny$z_3$};
\end{tikzpicture}
\label{subfig:d}
\end{subfigure}
\\
\vspace{4mm}
\begin{subfigure}[]{\textwidth}
\caption{$\phi_{p,q} \in \pi_2(p,q)$ where $p$, $q$ have $\alpha_g$ components  $x_{-2}$, $x_{1}$.}
\centering
\begin{tikzpicture}[scale=.65]
		\fill[gray!15]    (-9,2) -- (-9,0) -- (2,0) -- (4,2);
		\fill[gray!15]    (3,1) -- (2,0) -- (4,0) -- (5,1);
		\fill[gray!30]    (-9,2) -- (-9,0) -- (-2,0) -- (-.2,2);
		\fill[gray!30]    (-2,1) -- (-2,0) -- (2,0) -- (3,1);
		\fill[gray!30]    (-.5,2) -- (-.5,1.4) -- (1.4,1.7) -- (2,2);
		\fill[gray!30]    (-.5,2) -- (-.5,1.4) -- (-1.15,1) ;
		
		\draw [-, thick] (-9, 0) -- (11, 0);
		\draw [-, thick] (-9,2 ) -- (11, 2);
		\draw [-, red, thick] (-9, 1) -- (11, 1);
		
		\draw [-, blue, thick] (-8,0 ) -- (-6, 2);
		\draw [-, blue, thick] (-6,0 ) -- (-4, 2);
		\draw [-, blue, thick] (-4,0 ) -- (-2, 2);
		\draw [-, blue, thick] (0, 1.5)  parabola (-2,0 );
		\draw [-, blue, thick] (0,1.5 )  parabola  (2, 2);
		\draw [-, blue, thick] (2,0 ) -- (4, 2);
		\draw [-, blue, thick] (4,0 ) -- (6, 2);
		\draw [-, blue, thick] (6,0 ) -- (8, 2);
		\draw [-, blue, thick] (8,0 ) -- (10, 2);
		\draw [-, blue, thick] (11, .5)  parabola (10,0);
		\draw [-, blue, thick] (-9, 1.5)  parabola (-8,2);
		
		\draw [-, dotted, blue] (10,2) -- (10, 0);
		\draw [-, dotted, blue] (8,2) -- (8, 0);
		\draw [-, dotted, blue] (2,2) -- (2, 0);
		\draw [-, dotted, blue] (4,2) -- (4, 0);
		\draw [-, dotted, blue] (6,2) -- (6, 0);
		\draw [-, dotted, blue] (-2,2) -- (-2, 0);
		\draw [-, dotted, blue] (-4,2) -- (-4, 0);
		\draw [-, dotted, blue] (-6,2) -- (-6, 0);
		\draw [-, dotted, blue] (-8,2) -- (-8, 0);
		
		\draw [-, thick, orange] (-0.5, 0) -- (-0.5, 2);
		
		\filldraw (-1.15, 1) circle (1.5pt) node[below  ] {\tiny$x_{1}$};
		 \filldraw (-3, 1) circle (1.5pt) node[below  ] {\tiny$x_{2}$};
		 \filldraw (-5, 1) circle (1.5pt) node[below  ] {\tiny$x_{3}$};
		 \filldraw (-7, 1) circle (1.5pt) node[below  ] {\tiny$x_{4}$};
		 
		 \filldraw (3, 1) circle (1.5pt) node[below right] {\tiny$x_{-1}$};
		 \filldraw (5, 1) circle (1.5pt) node[below right] {\tiny$x_{-2}$};
		 \filldraw (7, 1) circle (1.5pt) node[below right] {\tiny$x_{-3}$};
		 \filldraw (9, 1) circle (1.5pt) node[below right] {\tiny$x_{-4}$};
		 
		 \filldraw (-.5, 1) circle (1.5pt) node[below right] {\tiny$x_0$};
		 \filldraw (-0.5, 1.4) circle (1.5pt) node[ right] {\tiny$\theta$};
		 
		  \node [below left] at (-8.2, 1.1) {\tiny$\alpha_g$};
		  \node [below right] at (9.8, 2) {\tiny$\beta_g$};
		  \node [ right] at (-.6, .2) {\tiny$\gamma_g$};
		  
		  \filldraw (-.2, 1.8) circle (.5pt) node[ right] {\tiny$w$};
		  \filldraw (-1.2, 1.8) circle (.5pt) node[ right] {\tiny$z$};
		  \filldraw (6.5, 1.5) circle (.5pt) node[ right] {\tiny$z_3$};
\end{tikzpicture}
\label{subfig:e}
\end{subfigure}
\caption{The domain of a disk $\phi_{p,q} \in \pi_2(p,q)$, for $p, q \in \bbT_\alpha \cap \bbT_\beta$ in the winding region can be identified with the domain of a disk $\phi \in \pi_2(\tilde{p}, \tilde{q}).$ }
\label{fig:domainsinwindingregion}
\end{figure}
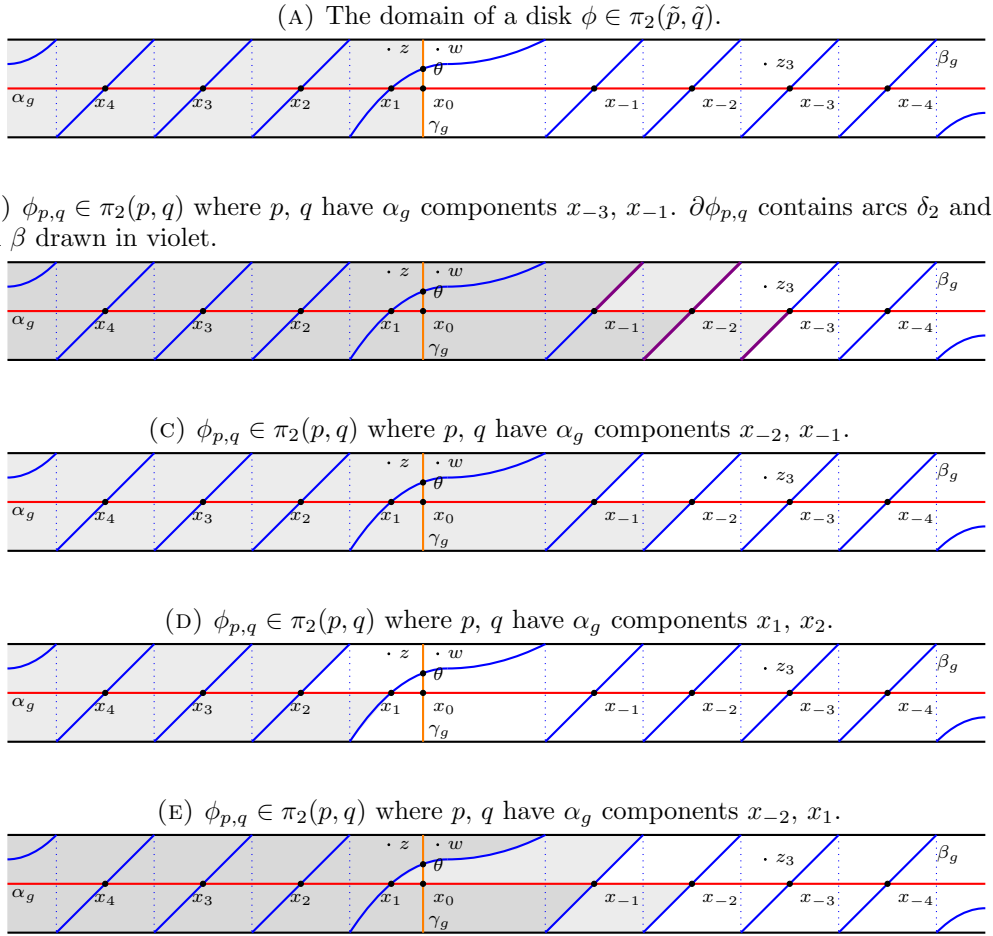

Theorem~\ref{hatfiltration} described the $\Z$-filtered chain homotopy type of knot Floer chain complex $\cfkhat(S^3_t(K), \mu_n, \s_m)$ for $t$ large with respect to $m$ and \textcolor{blue}{$n$}.  In Theorem~\ref{theorem:infinityfiltration}, we  describe the $\Z\oplus \Z$-filtered chain homotopy type of $\cfk^\infty(S^3_tK, \mu_n, \s_m)$. This generalizes Theorem 4.2 of Hedden-Kim-Livingston \cite{hedden-kim-livingston} which studies the $n=1$ case.

\begin{proof}[Proof of Theorem \ref{theorem:infinityfiltration}]
The isomorphism of chain complexes induced by the map (defined in \cite{HFK}) \[\Phi_m: \cf^\infty(S^3_t(K), \s_m) \to \cfk^\infty(S^3,K)\]
respects the $\F[U, U^{-1}]$-module structure of both complexes, and hence determines one of the $\Z$-filtrations (called the $U$-filtration) of $\cfk^\infty(S^3_{t}(K), \mu_n, \mathfrak{s}_m)$.

The knot $\mu_n \subset S^3_t(K)$ induces an additional $\Z$-filtration (the Alexander filtration) on $\widehat{\cf}(S^3_{t}(K), \mathfrak{s}_m)$ and on $\cfk^\infty(Y_t(K), \s_m)$. The additional $\Z$-filtration on $\cfk^\infty(Y_t(K), \mu_n, \s_m)$ can be determined in exactly the same way as it was determined for the case of $\widehat{\cf}(S^3_{t}(K), \mathfrak{s}_m)$. Lemma~\ref{lem:filtration} identifies the $\Z$-filtration induced on any given $i = constant$ slice in $\cf^\infty(S^3_t, \s_m)$ with a $(n+1)$-step filtration as above. This yields the statement of the theorem. 

Alternatively, the additional (Alexander) $\Z$-filtration on $\cfk^\infty(Y_t(K), \mu_n, \s_m)$ can be obtained from the Alexander filtration on $\cfkhat(Y_t(K), \mu_n, \s_m)$ by the fact that the $U$ variable decreases Alexander grading by one, i.e. we have the relation $A(U \cdot x) = A(x) - 1$. 
\end{proof}

\begin{figure}
\begin{tikzpicture}
\fill[gray!10]    (2,0) -- (8,6) -- (8,8) node[anchor= north east, black] {\tiny thick diagonal} -- (6,8) -- (0,2) -- (0,0);

\fill[gray!30]  (4.2, 2.2) rectangle (0, 1.8) node[anchor= south east, black] {\tiny$C\{\max(i,j-4)=0\} \to$};
\fill[gray!30]  (3.8, 0) rectangle (4.2, 2.2) ;

\draw[step=.5cm,gray, very thin] (0,0) grid (8,8);
\draw[thick,->] (0,4) -- (8,4) node[anchor=north west] {i};
\draw[thick,->] (4,0) -- (4,8) node[anchor=south east] {j};

\draw[red,-] (0,7.2) -- (4.2,7.2) -- (4.2,0) ;
\draw[red,-] (0,6.8) node[anchor=south east] {A = 0} -- (3.8,6.8) -- (3.8,0);

\draw[blue,-] (0,6.7) -- (3.7,6.7) -- (3.7,0) ;
\draw[blue,-] (0,6.3) node[anchor=south east] {A = -1} -- (3.3,6.3) -- (3.3,0);

\draw[violet,-] (0,6.2) -- (3.2,6.2) -- (3.2,0) ;
\draw[violet,-] (0,5.8) node[anchor=south east] {A = -2} -- (2.8,5.8) -- (2.8,0);

\draw[teal,-] (0,5.7) -- (2.7,5.7) -- (2.7,0) ;
\draw[teal,-] (0,5.3) node[anchor=south east] {A = -3} -- (2.3,5.3) -- (2.3,0);

\draw[purple,-] (0,5.2) -- (2.2,5.2) -- (2.2,0) ;
\draw[purple,-] (0,4.8) node[anchor=south east] {A = -4} -- (1.8,4.8) -- (1.8,0);

\draw[brown,-] (0,4.7) -- (1.7,4.7) -- (1.7,0) ;
\draw[brown,-] (0,4.3) node[anchor=south east] {A = -5} -- (1.3,4.3) -- (1.3,0);

\draw[cyan,-] (0,4.2) -- (1.2,4.2) -- (1.2,0) ;
\draw[cyan,-] (0,3.8) node[anchor=south east] {A = -6} -- (.8,3.8) -- (.8,0);

\draw[orange,-] (0,3.7) -- (.7,3.7) -- (.7,0) ;
\draw[orange,-] (0,3.3) node[anchor=south east] {A = -7} -- (.3,3.3) -- (.3,0);
\end{tikzpicture}
\caption{$\cfk^\infty(S^3,K)$ is supported along a thick diagonal of width $2g(K)+1$. The regions labeled $A =0, \dots, A = -7$  have constant Alexander grading $A_n'$ induced by $\mu_n$ on $\cf^\infty(S^3_t(K), \s_m)$. For spin$^c$ structures $\s_m$ where $|m| \leq g(K)$, sufficiently large surgery coefficient $t$, the algebraic filtration $i$ on $C\{\max(i,j-m)=0\}$ corresponds to the $\Z$-filtration induced by $\mu_n$ on $\cf^\infty(S^3_t(K), \s_m)$ where $n > 2g(K)$.}
\label{fig:thickdiagonal}
\end{figure}
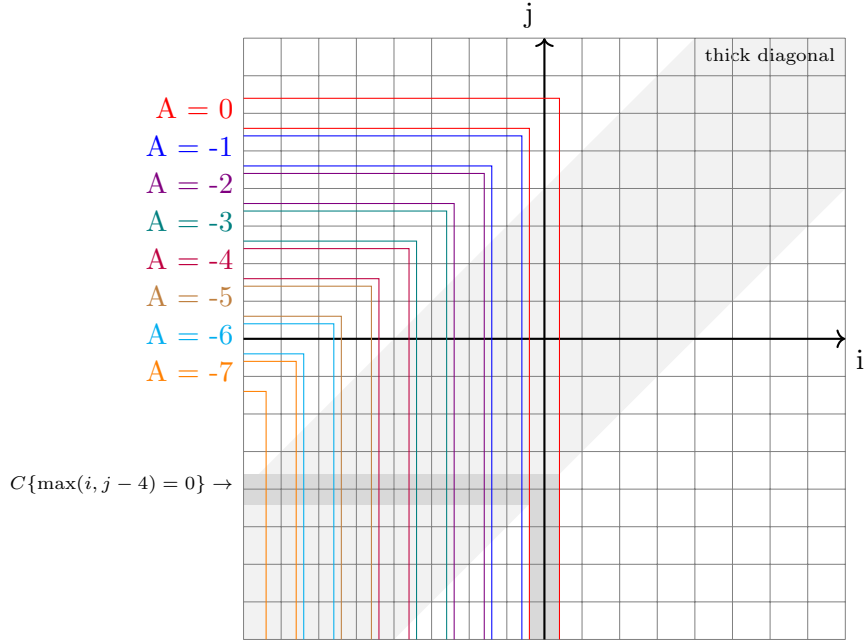

\begin{corollary} 
Let  $K$ be a knot in $S^3$ and fix $m, n \in \Z$. Then there exists $T = T(m, n) >0$ such that for all $t >T$ the following holds:
Up to a grading shift, the $p^{\text{th}}$ filtration level  of $\cfk^\infty(S^3_t(K), \mu_n, \s_m)$ is described in terms of the original $\mathbb{Z} \oplus \mathbb{Z}-$filtered knot Floer homology $\cfk^\infty(S^3, K)$ as 
\[ \text{max}(i, j - m - n)  = p .\]
\end{corollary}
That is, each Alexander filtration level $p$ of $\cfk^\infty(S^3_t(K), \mu_n, \s_m)$ is a ``hook" shaped region in $\cfk^\infty(S^3, K)$. 
\begin{proof} This follows from Theorem~\ref{theorem:infinityfiltration}.
\end{proof}

\begin{proposition}
Let $m \in \Z$ with $|m| \leq g(K)$ and let $n > 2g(K)$. For sufficiently large  surgery coefficient $t$, the Alexander filtration induced by $\mu_n$ on $\cf^\infty(S^3_t(K), \s_m)$ coincides with the algebraic $i$-filtration on $\cfk^\infty(S^3,K)$ under the correspondence given by $\Phi_m$.
\label{prop:thickdiagonal}
\end{proposition}

\begin{proof} Since $\cfkhat(Y,K)$ has degree equal to the Seifert genus of the knot, $\cfk^\infty(Y,K)$ is supported along a thick diagonal of width $2g(K)+1$. 
By the hypothesis, we have
\[m+n > g(K).\]
Therefore the corner $(p, m+n+p)$ of the hook region $C\{\max(i, j-m-n) = p\}$ of each constant Alexander filtration level $p$ of $\cfk^\infty(S^3_tK, \mu_n, \s_m)$ lies above the thick diagonal along which $\cfk^\infty(Y,K)$ is supported.  See Figure~\ref{fig:thickdiagonal}. 
For spin$^c$ structures $\s_m$ where $|m| \leq g(K)$, this means that the Alexander filtration induced by $\mu_n$ on $\cfk^\infty(S^3_t(K), \mu_n, \s_m)$ coincides with the algebraic $i$-filtration on $\cfk^\infty(S^3,K)$ under the correspondence given by $\Phi_m$.
\end{proof}

Because the algebraic $i$-filtration is used to define concordance invariants (such as $a_1(K)$, which can be interpreted as an integer lift of the Hom $\varepsilon$ invariant \cite{epsilon}), the filtration induced by $\mu_n$ on $\cf^\infty(S^3_t(K), \s_m)$ can be used to study the concordance class of a knot $K$. We will see that we can extract concordance invariants of $K$ from $\cfk^\infty(S^3_t(K), \mu_n, \s_m)$.

\section{A knot concordance invariant }
\label{section-epsilon}
As an application for the results in the previous section on the $\Z$--filtration induced on $\cfhat(S^3_{N}(K), \s_m)$ by the $(n,1)$--cable of the meridian $\mu_n$, our main result in this section (Theorem \ref{a1}) shows that the concordance invariant $a_1(K)$ of Hom \cite{Hom-concgroup}, which has an algebraic definition in terms of maps on subquotient complexes of $\cfk^\infty(K)$, can be equivalently defined by studying filtered maps on the (hat version of the) Heegaard Floer homology groups induced by the two-handle attachment cobordism of large integer surgery along a knot $K$ in $S^3$ and the filtration induced by the knot $\mu_n$ inside of the surgered manifold. 

Our result is analogous to the statement that the concordance invariants $\nu(K)$ of Ozsv\'ath-Szab\'o \cite{OzSzRational} and $\varepsilon(K)$ of Hom \cite{epsilon} can be defined algebraically or in terms of maps on the (hat version of the) Heegaard Floer homology groups induced by the two-handle attachment cobordism of large integer surgery along a knot $K$ in $S^3$. 
Definition \ref{def:epsilon} gives an algebraic definition of $\varepsilon(K)$ in terms of certain chain maps on the subquotient complexes of the knot Floer chain complex $\cfk^\infty(K)$. 
Due to the Ozsv\'ath-Szab\'o large integer surgery formula \cite{HFK}, $\varepsilon(K)$ can equivalently be defined in terms of maps on the Heegaard Floer chain complexes induced by the two-handle attachment cobordism of (large integer) surgery. 

We begin by recalling the definition of the concordance invariants $\varepsilon(K)$.   Let $N$ be a sufficiently large integer relative to the genus of a knot $K$. Consider the map 
\[ F_s : \hfhat(S^3) \to \hfhat(S^3_{-N}(K), [s]),\]
induced by the two-handle cobordism $W^4_{-N}$. Here, $[s]$ denotes the restriction to $S^3_{-N}(K)$ of the Spin$^c$ structure $\mathfrak{s}_s$ over $W^4_{-N}$ with the property that 
\[ \langle c_1(\mathfrak{s}_s), [\widehat{F}] \rangle - N = 2s,\]
where $| s | \leq \frac{N}{2}$ and $\widehat{F}$ denotes the capped off Seifert surface in the four-manifold. We also consider the map 
\[ G_s : \hfhat(S^3_N(K), [s]) \to \hfhat(S^3),\]
induced by the two-handle cobordism $-W^4_N$. 

The maps $F_s$ and $G_s$ can be defined algebraically by studying certain natural maps on subquotient complexes of $\cfk^\infty(K)$, as in \cite{HFK}. The map $F_s$ is induced by the chain map 
\[C\{i=0\} \to C\{\min(i,j-s) = 0\}\]
consisting of quotienting by $C\{i=0, j<s\}$ followed by the inclusion.
Similarly, the map $G_s$ is induced by the chain map
\[C\{\max(i,j-s) = 0\} \to C\{i=0\} \]
consisting of quotienting by $C\{i<0, j=s\}$ followed by the inclusion.

\begin{definition}[\cite{epsilon}, \cite{Hom-concgroup}] \label{def:epsilon}
Let $\tau = \tau(K)$ be the Ozsv\'ath-Szab\'o concordance invariant. 
The invariant $\varepsilon(K)$ is defined as follows:
\begin{itemize}
\item $ \varepsilon(K) = 1$ if $F_\tau$ is trivial (in which case $G_\tau$ is necessarily non-trivial).
\item  $ \varepsilon(K) = -1$ if $G_\tau$ is trivial (in which case $F_\tau$ is necessarily non-trivial).
\item $ \varepsilon(K) = 0$ if $F_\tau$ and $G_\tau$ are both non-trivial.
\end{itemize}
\end{definition}

In \cite{Hom-concgroup}, Hom defines a concordance invariant $a_1(K)$ for knots with $\varepsilon(K) = 1$ that is a refinement of $\varepsilon(K)$. 

\begin{definition}[\cite{Hom-concgroup}]
If $\varepsilon(K) = 1$ ($F_\tau$ is trivial), define
\[a_1(K) = \min\{s \ | \ H_s: H_*(C\{i = 0\}) \to H_*(C\{\min(i,j-\tau)=0, i\leq s\}) \text{ is trivial}\}\]
\end{definition}

We extend this definition of $a_1(K)$ to all knots (to include knots with $\varepsilon(K) \neq 1$). 
Consider the maps
\begin{eqnarray*}
G_{-s, \tau} : C\{\max(i,j-\tau) =0, i \geq -s\} \to C\{i=0\} 
\\
F_{s, \tau}: C\{i=0\}  \to C\{\min(i,j-\tau) =0, i \leq s\}
\end{eqnarray*}

\begin{definition}
Given a knot $K$ inside $S^3$, define: 
\begin{eqnarray*}
a_1(K) = 
\begin{cases} 
\max \{-s \ |\  G_{-s, \tau} \text{ is trivial on homology\} }, &\mbox{if } \varepsilon(K) = -1; \\ 
0, & \mbox{if } \varepsilon(K) = 0;\\
\min \{s \ |\ F_{s, \tau}  \text{ is trivial on homology\} }, &\mbox{if } \varepsilon(K) = 1. \\ 
\end{cases}
\end{eqnarray*}
\end{definition}
Note that $a_1(K)$ only depends on the doubly-filtered chain homotopy type of the knot Floer chain complex $CFK^{\infty}(K)$, so it is a knot invariant.

\begin{remark}
When $\varepsilon(K) = 1$, the definition of $a_1(K)$ agrees with the invariant $a_1(K)$ defined in Lemma 6.1  in \cite{Hom-concgroup}. As remarked in \cite{Hom-concgroup}, $a_1(K)$ measures the ``length" of the horizontal differential hitting the special class generating the vertical homology of $\cfhat(S^3)$. Similarly, when $\varepsilon(K) = -1$, $a_1(K)$ measures the ``length" of the horizontal differential coming out of the special class generating the vertical homology of $\cfhat(S^3)$.
\end{remark}

Recall that the rationally null-homologous knot $\mu_n \subset S^3_t(K)$ induces a $\Z$-filtration of $\cfhat(S^3_t(K), \s_\tau)$ and $\cfhat(S^3_{-t}(K), \s_\tau)$, that is, a sequence of subcomplexes:
\[0 \subset \mathcal{F}_{bottom} \subset \mathcal{F}_{bottom+1} \subset \cdots \subset \mathcal{F}_{top-1} \subset \mathcal{F}_{top} = \cfhat(S^3_t(K), \s_\tau).\]
\[0 \subset \mathcal{F}'_{bottom} \subset \mathcal{F}'_{bottom+1} \subset \cdots \subset \mathcal{F}'_{top-1} \subset \mathcal{F}'_{top} = \cfhat(S^3_{-t}(K), \s_\tau).\]
Using Theorem \ref{hatfiltration} and Proposition \ref{prop:thickdiagonal}, an equivalent definition of $a_1(K)$ can be formulated in terms of the filtration $\mathcal{F}$ and $\mathcal{F}'$ induced by $\mu_n $ as a knot inside $S^3_t(K)$ and $S^3_{-t}(K)$. This interpretation of the invariant $a_1(K)$ offers a  topological perspective that complements the original algebraic definition of $a_1(K)$.

\begin{theorem} 
Let $n > 2g(K)$.
For sufficiently large  surgery coefficient $t$, the concordance invariant $a_1(K)$ is equal to
\label{a1}
\[ a_1(K) = \begin{cases}
\max \left\lbrace m \mid  \begin{lrcases}
 \cfhat(S^3_tK,\ s_\tau) / \mathcal{F}_{ \text{top}-1-m} \to \cfhat(S^3) 
\\
  \text{ induces a trivial map on homology }   
 \end{lrcases} \right \rbrace & \text{ if } \varepsilon(K) = -1, \\
\\
0 & \text{ if } \varepsilon(K) = 0, \\\\
\min \left\lbrace m \mid  \begin{lrcases}
 \cfhat(S^3)  \to  \mathcal{F}'_{ \text{bottom}+m} \subset \cfhat(S^3_{-t}K, \ \s_\tau) 
\\
  \text{ induces a trivial map on homology }   
 \end{lrcases} \right \rbrace & \text{ if } \varepsilon(K) = 1. 
\end{cases}
\]

\end{theorem}
\begin{proof}
Since $|\tau| \leq g_4(K) \leq g(K)$, we can apply Proposition \ref{prop:thickdiagonal} which states that in the spin$^c$ structure $\s_\tau$, the algebraic $i$-filtration on $\cfk^\infty(S^3, K)$ coincides with the filtration induced by $\mu_n$ on $\cfhat(S^3_N(K), \s_\tau)$ under the identification of the two filtered chain complexes in Theorem \ref{hatfiltration}. 
\end{proof}

\begin{remark}
Recall that $a_1(K)$ is a concordance invariant (see Proposition \ref{a1conc}) that fits into a family of concordance invariants studied by Dai, Hom, Stoffregen and the author in \cite{DHSTCFK}. It would be interesting to see if an analogue of Theorem \ref{a1} exists for this entire family of algebraically defined invariants corresponding to the standard local representative (over $\F[U,V]/(UV)$) of the knot. 
\end{remark}

\begin{proposition}[\cite{Hom-concgroup}]
\label{a1conc}
The invariant $a_1(K)$ is a concordance invariant. 
\end{proposition}
\begin{proof}
Suppose $K_1$ and $K_2$ are concordant knots, i.e. $K_1 \# \overline{K_2}$ is slice. Then $\varepsilon( K_1 \# \overline{K_2} ) = 0$. By Proposition 3.11 in \cite{hom-survey}, we may find a basis for $\cfk^\infty(K_1 \# \overline{K_2})$ with a distinguished element $x$ that generates the homology $\hfk^\infty(K_1 \# \overline{K_2})$ and splits off as a direct summand of $\cfk^\infty(K_1 \# \overline{K_2})$.  Similarly, we can find a basis for $\cfk^\infty(K_2 \# \overline{K_2})$ with a distinguished element $y$ with the same properties. Then to compute $a_1 (K_2 \# K_1 \# \overline{K_2})$, by the Kunneth principle \cite{HFK} we can consider either chain complex:
\begin{eqnarray*}
\cfk^\infty(K_1 \# \overline{K_2}) \otimes_{\mathbb{Z}[U, U^{-1}]} \cfk^\infty (K_2) 
&\text{or}&
\cfk^\infty (K_1) \otimes_{\mathbb{Z}[U, U^{-1}]}  \cfk^\infty(K_2 \# \overline{K_2}).
\end{eqnarray*}
Using the special bases from above, the relevant summands to $a_1$ are
\begin{eqnarray*}
 \{x\} \otimes \cfk^\infty (K_2)   & \text{ or } &   \cfk^\infty (K_1) \otimes \{y\}.
 \end{eqnarray*}
 Thus, $a_1(K_2) = a_1 (K_2 \# K_1 \# \overline{K_2}) = a_1(K_1)$.
\end{proof}

\example[Homologically thin knots]
Model complexes for $\cfk^\infty$ of homologically thin knots are studied in \cite{petkova}. Petkova shows that if $\tau(K) = n$, the model complex contains a direct summand isomorphic to 
\begin{eqnarray*}
\cfk^\infty(T_{2,2n+1}) \text{ if } n > 0 & \text{and} & \cfk^\infty(T_{2,2n-1}) \text{ if } n < 0.
\end{eqnarray*}
This summand supports $H_*(\cfk^\infty(K))$ and thus determines the value of $a_1(K)$. It is easy to see from the complex that $a_1(K) = \sgn(\tau(K))$. 

\begin{proposition}
\label{prop:tilde}
The following are properties of $a_1(K)$:

\begin{enumerate}
\item[(1)]
If $K$ is smoothly slice, then $a_1(K)=0$.

\item[(2)]
$\sgn(a_1(K)) = \varepsilon(K)$. 

\item[(3)]
$a_1(K) = - a_1(\overline{K})$.

\item[(4)] \label{connect-sum-tilde-epsilon}
If $a_1(K) = 0$, then $a_1(K\#K') = a_1(K')$.

\end{enumerate}

\end{proposition}

\begin{proof}[Proof of (1)] 
If $K$ is smoothly slice, then $\varepsilon (K) = 0$; therefore, $a_1(K) = 0$. 
\end{proof}

\begin{proof}[Proof of (2)] 
By construction, if $a_1(K) > 0 $, then $\varepsilon(K) = 1$; if $a_1(K) < 0$, then $\varepsilon(K) = -1$.

If $a_1(K) = 0$, we show that $\varepsilon(K) = 0$. Suppose $\varepsilon(K) = -1$. Then the vanishing of 
\[a_1 (K) = \max \{n \ |\ G_{n,\tau}  \text{ is trivial on homology\} } \]
implies that the map $G_{0,\tau}:C\{ i=0, j\leq \tau  \} \to  C\{i=0\}  $ is trivial on homology, which contradicts the definition of $\tau$. Similarly, $\varepsilon(K) \neq 1$ if $a_1(K) = 0$. 

Finally, according to \cite{epsilon}, $\varepsilon(K) = 0$ implies that $\tau(K) = 0 $.
\end{proof}

\begin{proof}[Proof of (3)] 
The symmetry properties of $\cfk^\infty$ of Section 3.5 in \cite{HFK} imply that $a_1(K) = - a_1(\overline{K})$.
\end{proof}

\begin{proof}[Proof of (4)] 
If $a_1(K) = 0$, $\varepsilon( K) = 0$. By Lemma 3.3 from \cite{epsilon}, we may find a basis for $\cfk^\infty(K)$ with a distinguished element $x$ which is the generator of both vertical and horizontal homology. Then $a_1(K \# K') $ can be computed from $ \{x\} \otimes \cfk^\infty(K')$.
\end{proof}

In fact, we can extend Proposition~\ref{connect-sum-tilde-epsilon}(4) to describe the behavior of $a_1$ under connect sum in many (but not all) cases. 
\begin{proposition}
\label{connectsumextended}


\item[(1)]
If $a_1(K_1) > 0$ and $a_1(K_2) < 0$ and ${a_1(K_1)} + {a_1(K_2)} < 0$, then \[a_1(K_1 \# K_2) =  a_1(K_1).\]
\item[(2)]
If $a_1(K_1) > 0$ and $a_1(K_2) < 0$ and ${a_1(K_1)} + {a_1(K_2)} > 0$, then \[a_1(K_1 \# K_2) =  a_1(K_2).\]
\item[(3)]
If $a_1(K_1) > 0$ and $a_1(K_2) > 0$, then $a_1(K_1 \# K_2) = \min(a_1(K_1), a_1(K_2))$.
\item[(4)]
If $a_1(K_1) < 0$ and $a_1(K_2) < 0$, then $a_1(K_1 \# K_2) = \max(a_1(K_1), a_1(K_2))$.
\end{proposition}

\begin{proof}
Note that we use $-K$ to denote the mirror of a knot $K$. 

\item[(1)] See the proof of Lemma 6.3 of \cite{Hom-concgroup}.

\item[(2)] 
The mirrors $-K_1$ and $-K_2$ satisfy the hypothesis of (1), so 
\[a_1(-K_1 \# -K_2) = a_1(-K_2).\]
Apply the symmetry property of $a_1$ under mirroring (\ref{prop:tilde}):
\[-a_1(K_1 \# K_2) = - a_1(K_2)\]
as desired. 

\item[(3)]
By Lemma 6.2 of \cite{Hom-concgroup}, there exists a basis $\{x_i\}$ over $\mathbb{F}[U, U^{-1}]$ for $\cfk^\infty(K_1)$ with basis elements $x_0$ and $x_1$ with the property that
\begin{enumerate}
\item There is a horizontal arrow of length $a_1$ from $x_1$ to $x_0$.
\item There are no other horizontal arrows or vertical arrows to or from $x_0$.
\item There are no other horizontal arrows to or from $x_1$.
\end{enumerate}
Similarly, we may find a basis $\{y_i\}$ over $\mathbb{F}[U, U^{-1}]$ for $\cfk^\infty(K_2)$ with basis elements $y_0$ and $y_1$ with the above properties. Without loss of generality, assume that $a_1(K_1) \leq a_1(K_2)$.

Notice $x_0y_0$ generates the vertical homology $H_*(C(\{i=0\}))$ of $\cfk^\infty(K_1 \# K_2)$. Let $\tau = \tau(K_1 \# K_2)$.
Consider the subquotient complex 
\[ A = C\{\min(i,j-\tau) = 0\}.\]
There is a direct summand of $A$ consisting of the generators $x_0y_0$, $x_0y_1$, $x_1y_0$, and $x_1y_1$, and four horizontal arrows as shown in Figure \ref{fig:a}.
The arrow $x_1y_0$ to $x_0y_0$ has length $a_1(K_1)$. Clearly, $\varepsilon(K_1 \# K_2) = 1$ and $a_1(K_1\#K_2) = a_1(K_1)$.
\begin{figure}[!h]
\begin{tikzpicture}
	 \useasboundingbox (-2, -0.5) rectangle (10.5, .5);
        \filldraw (0, 0) circle (1.5pt) node[] (x0y0){};
        \draw (0, .2) node[] (){$x_0y_0 $};
        \filldraw (3, .1) circle (1.5pt) node[] (x1y0){};
        \draw (3, .3)  node[] (){$ x_1y_0$};
        \filldraw (7, -.1) circle (1.5pt) node[] (x0y1){};
        \draw (7, -.3) node[] (){$ x_0y_1$};
        \filldraw (10, 0) circle (1.5pt) node[] (x1y1){};
        \draw (10, .2) node[] (){$ x_1y_1$};
	\draw [->] (x0y1) -- (x0y0);
	\draw [->] (x1y0) -- (x0y0);
	\draw [->] (x1y1) -- (x1y0);
	\draw [->] (x1y1) -- (x0y1);
\end{tikzpicture}
\caption[]{A direct summand of $A=C\{\min(i,j-\tau) = 0\}$ in Proposition \ref{connectsumextended}(3)} This is the summand that is relevant for computing $a_1$, as it contains the generator $x_0y_0$ of vertical homology $H_*(C\{i=0\})$.
\label{fig:a}
\end{figure}
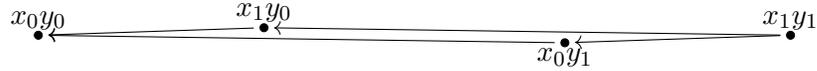

\item[(4)] The mirrors $-K_1$ and $-K_2$ satisfy the hypothesis of (3). So
\begin{eqnarray*}
-a_1(K_1\# K_2) =a_1(-K_1\# -K_2) = \min(a_1(-K_1), a_1(-K_2)) = \min(-a_1(K_1), -a_1(K_2)) \\
= -\max(a_1(K_1), a_1(K_2)).
\end{eqnarray*}
\end{proof}

Proposition~\ref{connectsumextended} can be rewritten as the following. 

\begin{proposition}
If $a_1(K_1) \neq 0 $ and $a_1(K_2)\neq 0$:
\begin{enumerate}
\item[(1)]
If 
$a_1(K_1) + a_1(K_2) < 0 $, then 
$a_1(K_1 \# K_2) = \max (a_1(K_1), a_1(K_2))$.
\item[(2)]
If ${a_1(K_1)} + {a_1(K_2)} > 0$, then $a_1(K_1 \# K_2) = \min (a_1(K_1), a_1(K_2))$.
\end{enumerate}
\end{proposition}

\begin{remark}
If $a_1(K) \neq 0 $ and $a_1(K')\neq 0$, 
and 
$a_1(K) + a_1(K') = 0 $, then $a_1(K\#K') $ is indeterminate. The next two examples illustrate this case.
\end{remark}

\example 
The connect sum of any knot $K$ with the reverse of its mirror $-{K}$, i.e. the inverse of $K$ in the concordance group $\mathcal{C}$, has vanishing $a_1(K \# -{K}) = 0$.

\example
The full knot Floer chain complexes $\cfk^\infty$ of the mirror $-T_{2,3; 2,5}$ of the $(2,5)$-cable of the torus knot $T_{2,3}$, the torus knot $T_{2,9}$, and the connect sum $-T_{2,3; 2,5} \# T_{2,9}$ are described in \cite{Hom-Wu}. It is easy to see that
$a_1(-T_{2,3; 2,5}) = -1$, $a_1(T_{2,9}) = 1$, and 
$a_1(-T_{2,3; 2,5} \# T_{2,9}) = -1$.
\\

We conclude with some computations of the $a_1$--invariant. 
\example
In \cite{homnote} Hom produces the relevant summand of $\cfk^\infty$ for computing $\varepsilon$ and hence $a_1$ for the knot $T_{4,5} \# - T_{2,3; 2,5}$. It is easy to see that $a_1(T_{4,5} \# - T_{2,3; 2,5}) = 2$.

\example
The Conway knot $C_{2,1}$ has $a_1(C_{2,1}) = 0$. According to \cite{peters}, the knot Floer chain complex $\cfk^\infty(C_{2,1})$ is generated as a $\mathbb{F}[U,U^{-1}]-$module by a single isolated $\mathbb{F}$ at the origin plus a collection of null-homologous ``boxes".

\example
The knot Floer chain complex of an L-space knot is a given by Theorem 2.1 in \cite{Upsilon}. If $K$ is an L-space knot, with Alexander polynomial 
\begin{eqnarray*}
\Delta_K(t) =  \sum_{i=0}^k (-1)^i t^{n_i},
\end{eqnarray*}
where $n_0 > n_1 > \cdots > n_k$,
then $a_1(K) = n_0-n_1$ by Lemma 6.5 \cite{Hom-concgroup}. 

\bibliographystyle{amsalpha}
\bibliography{references}

\end{document}